\documentclass{amsart}
\usepackage[usenames,dvipsnames]{xcolor}
\definecolor{darkgreen}{rgb}{0,0.5,0}
\usepackage{url}
\usepackage{amsthm}%
\usepackage[
colorlinks, citecolor=darkgreen,
backref,
pdfauthor={Ander Martin},
pdftitle={}
]{hyperref} 
\hypersetup{
	colorlinks=true,
	linkcolor=blue,
	citecolor=blue,
	urlcolor=blue,
	linktoc=page
}
\usepackage{tikz,enumerate,caption,stmaryrd,amsfonts,amssymb,systeme,comment}
\usetikzlibrary{matrix,positioning,arrows,shapes,chains,calc,automata}
\usetikzlibrary{decorations.pathreplacing}
\usepackage[T1]{fontenc}

\newlength{\proofmargin}
\renewcommand*{\backref}[1]{}
\renewcommand*{\backrefalt}[4]{%
	\ifcase #1 %
	\relax
	\or
	$\uparrow$#2.%
	\else
	$\uparrow$#2.%
	\fi%
}

\DeclareFontFamily{U}{wncy}{}
\DeclareFontShape{U}{wncy}{m}{n}{<->wncyr10}{}
\DeclareSymbolFont{mcy}{U}{wncy}{m}{n}
\DeclareMathSymbol{\Sha}{\mathord}{mcy}{"58}

\DeclareMathOperator{\Q}{\mathbf{Q}}

\DeclareMathOperator{\Z}{\mathbf{Z}}
\DeclareMathOperator{\C}{\mathbf{C}}

\DeclareMathOperator{\et}{\mbox{\scriptsize \'et}}

\DeclareMathOperator{\dR}{\mbox{\scriptsize dR}}

\DeclareMathOperator{\cris}{\mbox{\scriptsize cris}}

\DeclareMathOperator{\Hom}{\mathrm{Hom}}

\DeclareMathOperator{\Gal}{\mathrm{Gal}}
\DeclareMathOperator{\Res}{\mathrm{Res}}

\DeclareMathOperator{\Tr}{\mathrm{Tr}}

\DeclareMathOperator{\loc}{\mathrm{loc}}

\newcommand{\Ker}{\mathrm{Ker}}

\DeclareFontFamily{U}{wncy}{}
\DeclareFontShape{U}{wncy}{m}{n}{<->wncyr10}{}
\DeclareSymbolFont{mcy}{U}{wncy}{m}{n}
\DeclareMathSymbol{\Sh}{\mathord}{mcy}{"58} 

\newtheorem{theorem}{Theorem}[section]
\newtheorem{prop}[theorem]{Proposition}

\newtheorem{lemma}[theorem]{Lemma}

\theoremstyle{plain}
\newtheorem*{unnumberedtheorem}{Theorem}
\usepackage{tikz-cd}
\theoremstyle{plain}
\newtheorem{remark}{Warning}

\begin{document}
	\title{Lie weights for CM curves and unipotent Selmer dimensions}
	\author{Ander Martin}
	\thanks{%
		\parbox[t]{\dimexpr\linewidth-\parindent\relax}{%
			\raggedright
			\linespread{1}\fontsize{9}{10}\selectfont
			School of Mathematics, University of Edinburgh, Edinburgh, UK\newline
			\texttt{ander.martin-iribar@ed.ac.uk}%
			\newline
			\linespread{1}\fontsize{9}{10}\selectfont
			ORCID: 0009-0006-9176-0073
		}\nopunct
	}
	\begin{abstract}
		We prove an equidistribution result for Galois characters in the graded Lie components $L_n$ associated to the $\Q_p$-fundamental group of a projective hyperbolic curve $X$ over a number field $k$ with geometrically CM Jacobian. That is, we consider the $p$-adic completion of a fixed character lattice, where we define measures $\mu_n$ counting normalized character multiplicities in $L_n$, and show that $(\mu_n)_n$ weakly converges to the Haar measure. The key idea to prove convergence is to write the Fourier coefficients of $\mu_n$ in terms of traces of operators on $L_n$ induced by symplectic operators on the Tate module, and control them with a known symplectic character formula, followed by Weyl's criterion for compact abelian groups. Our main arithmetic application is to combine equidistribution with the Iwasawa-theoretic method of Coates--Kim and a compactness argument for moving analytic loci, giving a dimension inequality for unipotent Selmer schemes: $\dim_{\Q_p}H^1_f(G_{k,T},U_n)
		<
		\dim_{\Q_p}\prod_{v\mid p}H^1_f(G_v,U_n)$ for all sufficiently large $n,$ with the quotient having limit superior at most $1/2.$ This answers a particular case of a conjecture by Kim, from which the finiteness of $X(\Q)$ can be recovered when $k=\Q.$
	\end{abstract}
	\maketitle
	\tableofcontents
	\newpage
	\section{Introduction}
	Let $X/k$ be a smooth, projective, geometrically integral curve of genus
	$g\geq2$ over a number field, with a point $b\in X(k)$ and geometrically
	CM Jacobian $J$. Fix an odd prime $p$ such that $X$ has good reduction
	at every place above $p$. Consider the étale $\Q_p$-pro-unipotent
	fundamental group $U:=\pi_1^{\et,\Q_p}(X_{\bar k},b)$ of $X$ with its lower central series $U^1=U$, $U^{n+1}=[U,U^n].$ Let
	$$
	U_n=U/U^{n+1},\qquad L_n=U^n/U^{n+1},\qquad
	\mathrm{gr}(U)=\bigoplus_{n\geq1}L_n.
	$$
	Here $\mathrm{gr}(U)$ is the associated graded Lie algebra, with brackets given by commutators. Write $d_n=\dim_{\Q_p}L_n$ and
	$D_n=\sum_{j=1}^n d_j=\dim_{\Q_p}U_n$. The CM hypothesis allows us to make the action on the Tate module $L_1\simeq V_pJ$ explicit.
	After choosing finite extensions $F/k$ and $L/\Q_p$ as in Section \ref{class},
	we can write $V_L:=V_pJ\otimes_{\Q_p}L\simeq\bigoplus_{i=1}^{2g}L(\chi_i)$
	as $G_F$-representations, where each character $\chi_i$ takes values in
	$1+p\mathcal O_L$. Since $\mathrm{gr}(U)$ is generated by $L_1$,
	there is a $G_F$-equivariant quotient map
	$$
	V_L^{\otimes n}\twoheadrightarrow L_n\otimes_{\Q_p}L,
	\qquad
	v_1\otimes\cdots\otimes v_n
	\longmapsto[v_1,[v_2,\ldots,[v_{n-1},v_n]\ldots]].
	$$
	A bracket of
	eigenvectors with characters $\chi_{i_1},\ldots,\chi_{i_n}$ has their product
	$\prod_{j=1}^n\chi_{i_j}$ as its character. Writing
	$L_n\otimes_{\Q_p}L\simeq\bigoplus_\theta L(\theta)^{m_n(\theta)}$,
	the multiplicity $m_n(\theta)$ is the dimension of the subspace spanned
	by brackets in $L_n \otimes_{\Q_p} L$ that have $\theta$ as their character. We will show that these multiplicities/dimensions become uniformly distributed asymptotically.
	
	Purity of the Tate module allows us to assign degree $1$ to each $\chi_i,$ so every character in $L_n \otimes_{\Q_p} L$ has degree $n.$ However, as we prove in Warning \ref{warning}, one must remove a common factor in order to compare different degrees. We do that by setting $u_i=\chi_1^{-1}\chi_i$ and
	$B=\langle u_1,\ldots,u_{2g}\rangle$. Every character in $L_n$ can then be written as $\chi_1^n u$ with $u \in B.$ Since $B$ kills the degree, we are ready to compare characters by varying the degree now! On the completion $B_p=\varprojlim_k B/B^{p^k}\simeq\Z_p^a$, we define the probability measures
	\begin{equation}\label{mu}
		\mu_n=\frac{1}{d_n}\sum_{u\in B}m_n(\chi_1^n u)\delta_u,
	\end{equation}
	where $\delta_u$ is the Dirac measure at $u$. They keep track of each character $u$'s multiplicity relative to the total one in degree $n$. Let $\mu$ be the Haar probability measure on $B_p$. We prove the following equidistribution result.
	\begin{unnumberedtheorem}[Theorem~\ref{conv}]
		The measures $\mu_n$ weakly converge to $\mu$.
		In particular, for every $k\geq1$ and $r\in B_p$,
		we have, for the projection
		$\pi_k:B_p\rightarrow B/B^{p^k}$,
		$$\mu_n\left(\pi_k^{-1}\left(\pi_k(r)\right)\right)
		=
		\frac{1}{d_n}
		\sum_{\substack{u\in B\\\pi_k(u)=\pi_k(r)}}
		m_n(\chi_1^n u)
		\xrightarrow{n\to\infty}p^{-ak}.$$
	\end{unnumberedtheorem}
	In other words, at every fixed $p$-adic precision, the residue classes in $B_p$
	carry asymptotically equal proportions of the dimension. 
	
	To prove this, we integrate continuous characters $\eta:B_p \rightarrow S^1$ against $\mu_n,$ obtaining the Fourier coefficients $\int_{B_p}\eta\,d\mu_n.$ We express them as certain normalized
	traces as follows. After choosing an embedding
	$L\hookrightarrow\C$, $\eta$ induces a symplectic operator on $V_L \otimes_L \C,$ which at the same time gives rise to another operator on $L_n \otimes_{\Q_p} \C;$ see (\ref{comp}). The symplectic character formula
	\cite[Proposition~4.4]{Fil13} bounds these traces by $o(d_n)$
	(Lemma \ref{trace}), and Weyl's criterion for compact abelian groups gives the required convergence.
	
	The arithmetic application follows the Iwasawa-theoretic strategy of
	Coates--Kim \cite{CK10}. We fix an Iwasawa class module $Y_R$ in Section \ref{est}, which is
	finitely generated and torsion over a localized power series ring $R$, and we can choose
	a non-zero annihilator $\delta$. The relation $\delta Y_R=0$ forces
	every character seen by $R$-linear maps
	$Y_R\to L_n^\vee(1)\otimes_{\Q_p}L$ to satisfy a $p$-adic analytic
	equation, and we can bound their normalized total multiplicity by
	$$
	\mu_n\{u\in B_p:f(u,n)=0\},
	$$
	where $f$ is a $p$-adic analytic function obtained by twisting $\delta,$ and is non-zero as we show in
	Lemma \ref{nonzero}. Such functions have loci of Haar
	measure zero, but here the zero set also depends on $n$, so we are not only dealing with convergence of measures, but also with moving loci! A compactness
	argument (Lemma \ref{slice}) combined with equidistribution fixes this issue, and the aforementioned multiplicities tend to zero. Consequently, we get the main Iwasawa estimate
	$\dim_L\Hom_R(Y_R,L_n^\vee(1)\otimes_{\Q_p}L)=o(d_n)$
	(Theorem \ref{estimate}).
	
	To interpret this arithmetically, let $T$ be a finite set containing the
	archimedean places of $k$, the ones above $p$ and those of bad
	reduction. Let $G_{k,T}=\Gal(k_T/k)$, where $k_T$ is the maximal
	extension of $k$ unramified outside of $T$, and
	$G_v=\Gal(\overline{k_v}/k_v)$.
	The Bloch--Kato Selmer scheme\footnote{We refer the reader to \cite{Bet23} and \cite{Kim05} for a discussion on the representability of these cohomology functors.} $H^1_f(G_{k,T},U_n)$ imposes crystallinity at every $v\mid p.$ Let $U_{n,v}^{\dR}$ be the depth-$n$ quotient of the
	de Rham fundamental group\footnote{Following \cite{Kim09}, we use Wojtkowiak's filtration \cite[Section~5.6]{Woj93} on the de Rham fundamental group. Betts notes \cite[Remark~4.2.2]{Bet23} that its agreement with Hadian's filtration is expected, but not published.} of $X_{k_v}.$
	The Bloch--Kato logarithm identifies $H^1_f(G_v,U_n) \simeq \Res_{k_v/\Q_p}U_{n,v}^{\dR}/F^0$ as affine $\Q_p$-schemes
	\cite[Proposition~3.1.10]{Bet23}. We prove the following result.
	\begin{unnumberedtheorem}[Theorem~\ref{thm}]
		(Projective CM case)
		As $n\to\infty$,
		\begin{align*}
			\dim_{\Q_p}H^1_f(G_{k,T},U_n)
			&\leq\frac{[k:\Q]}{2}D_n+o(D_n),\\
			\dim_{\Q_p}\prod_{v\mid p}
			\Res_{k_v/\Q_p} U_{n,v}^{\dR}/F^0
			&=[k:\Q]D_n-o(D_n).
		\end{align*}
		In particular, for all $n\gg0$,
		$$\dim_{\Q_p}H^1_f(G_{k,T},U_n)
		<
		\dim_{\Q_p}\prod_{v\mid p}
		\Res_{k_v/\Q_p} U_{n,v}^{\dR}/F^0,
		\qquad
		\limsup_{n\to\infty}
		\frac{\dim_{\Q_p}H^1_f(G_{k,T},U_n)}
		{\dim_{\Q_p}\prod_{v\mid p}
			\Res_{k_v/\Q_p} U_{n,v}^{\dR}/F^0}
		\leq\frac12.$$
	\end{unnumberedtheorem}
	For $k=\Q$, this proves the projective CM case of
	\cite[Conjecture~1]{Kim09}, from which the finiteness of $X(\Q)$ follows, as shown in \cite[p.~121]{Kim09}. In \cite{CK10}, the inequality is proved over $\Q$ for the meta-abelian quotient $W_n$ of $U_n,$ under a splitting hypothesis on $p.$ The main additional difficulty for $U_n$ is that the dimensions of $L_n$ grow exponentially, and their argument using Hall bases and symmetric powers is no longer sufficient. This is precisely where equidistribution comes into play.
	
	The connection with cohomology comes from restricting locally trivial cocycles to obtain the $R$-linear maps from $Y_R$ (Proposition \ref{asy}).
	Poitou--Tate duality then gives $h^2(G_{k,T},L_n)=o(d_n)$. The global Euler characteristic formula, together with the same trace estimate as before (applied to complex conjugation at real places), gives the global estimate. Hodge filtration estimates give the local one.

	We show equidistribution in Section \ref{equid}, deduce the Iwasawa estimate in Section \ref{est}, and prove the cohomological estimates and the dimension inequality in Section \ref{cohom}.

	\subsection*{Acknowledgements} I would like to thank Minhyong Kim for many helpful discussions and suggestions, as well as for encouraging me to continue working on this paper. I am also grateful to Xiang Li for reading a draft and providing feedback. This work was supported by the UKRI Centre for Doctoral Training in Algebra, 
	Geometry and Quantum Fields (AGQ), Grant Number EP/Y035232/1.
	\section{CM Lie weights}\label{equid}
	\subsection{Setup}\label{class}
	We define the field extensions and characters mentioned in the introduction\footnote{This subsection is very similar to \cite[Section~1]{CK10}. We spell out the constructions because we will say more about the $\iota_i$ in Lemma \ref{rank}}.
	Since $J$ is CM, there exists a commutative semisimple algebra
	$$K=\prod_{i=1}^r K_i \subset \mathrm{End}^0\left(J_{\overline{k}}\right), \qquad \dim_{\Q} K = 2g,$$
	where each $K_i$ is a CM field. Let
	$$K_p = K \otimes_{\Q} \Q_p \simeq \prod_{i=1}^r \prod_{\mathfrak{p} \mid p} K_{i, \mathfrak{p}}, \qquad \mathcal{O}_K = \prod_{i=1}^r \prod_{\mathfrak{p}\mid p} \mathcal{O}_{K_i, \mathfrak{p}},$$
	and notice that $V_pJ$ is a free $K_p$-module of rank $1.$ Let $F_0/k$ be a finite Galois extension such that the endomorphisms in $K$ are defined over $F_0$ (hence the action of $G_{F_0}$ on $V_pJ$ commutes with that of $K$) and $\mu_p \subset F_0.$ We then have
	$$\rho_0: G_{F_0} \rightarrow \operatorname{Aut}_{K_p} V_pJ \simeq K_p^\times,$$
	and $\rho_0(G_{F_0})$ is a compact abelian group.
	Since $p$ is odd, the $p$-adic logarithm gives an isomorphism of topological groups
	$$\log: 1+ p \mathcal{O}_K \xrightarrow{\sim} p \mathcal{O}_K.$$ Replace $F_0$ and $\rho_0$ by
	$$F:=\left(\overline{k}\right)^{\rho_0^{-1}(1+p \mathcal{O}_K)}, \qquad \rho:=\rho_0|_{G_F}.$$
	Then $\rho(G_F) \subset 1+p\mathcal{O}_K,$ and $\log (\rho (G_F))$ is a compact subgroup of the finitely generated $\Z_p$-module $p \mathcal{O}_K.$ Thus $\log (\rho(G_F))$ is a free $\Z_p$-module of finite rank $c$ for some $c \in \Z_{\geq 0}.$ The logarithm therefore identifies $\rho(G_F)$ with $\Z_p^c$. Define
	$$F_\infty:=F \left(J\left[p^\infty\right]\right) \ \text{and} \ \Gamma:=\Gal(F_\infty/F).$$ An element of $G_F$ fixes $F_\infty$ exactly when it acts trivially on $T_pJ$, or equivalently on $V_pJ$. Thus $$F_\infty = \overline{F}^{\ker \rho} \ \ \ \text{and} \ \ \  \Gamma \simeq G_F / \ker \rho \simeq \rho(G_F) \simeq \Z_p^c.$$ The Weil pairing $J[p^n] \times J[p^n] \rightarrow \mu_{p^n}$ is surjective and Galois-equivariant, so $F(\mu_{p^\infty}) \subset F_\infty,$ so $\Gamma$ is infinite and $c\geq1$. Let $L/\Q_p$ be a finite extension splitting $K_p,$ so that $K_p \otimes_{\Q_p} L \simeq L^{2g}.$ Denote the corresponding $\Q_p$-algebra maps $ \iota_1, \dots, \iota_{2g}: K_p \rightarrow L.$ These maps give the $2g$ continuous characters
	$$\chi_i = \iota_i \circ \rho: \Gamma \xrightarrow{\rho}1+p\mathcal{O}_K \subset (K_p)^\times \xrightarrow{\iota_i} 1+p \mathcal{O}_L \subset L^\times.$$
	The identification $V_pJ\simeq K_p$ as $K_p$-modules gives the following isomorphisms as $G_F$-representations.
	$$ V_L \simeq \bigoplus_{i=1}^{2g} L(\chi_i).$$
	Here $L(\chi_i)$ means the $1$-dimensional $L$-representation with the $\Gamma$-action given by the character $\chi_i$ (up to $g$ characters can coincide).
	
	Let $T'$ be the union of $T$ and the places of $k$ that ramify in $F/k,$ and let $S$ be the set of places of $F$ above the places of $T'.$ Since $J$ has good reduction outside $T,$ and $\mu_{p^\infty}$ is unramified away from the places above $p,$ it follows that $F_\infty/F$ is unramified outside $S.$
	
	\subsection{Measures and equidistribution}
	We explain our choice of multiplicity-counting measures and show their weak convergence to the Haar measure. 
	
	The Weil pairing
	$V_L\times V_L\to L(1)$ pairs $V_L(\chi)$ perfectly with
	$V_L(\chi_{\mathrm{cyc}}\chi^{-1})$ and trivially with all other
	eigenspaces. Fix a generator of $L(1)$, dual bases in the paired
	eigenspaces, and a symplectic basis in the eigenspaces where
	$\chi^2=\chi_{\mathrm{cyc}}$. This gives a symplectic eigenbasis
	$$
	e_{x_1},\ldots,e_{x_g},e_{y_1},\ldots,e_{y_g},
	\qquad \chi_{x_i}\chi_{y_i}=\chi_{\mathrm{cyc}},
	$$
	also denoted by $e_1,\ldots,e_{2g}$. In particular,
	$\chi_{\mathrm{cyc}}$ takes values in $1+p\mathcal O_L$. Write $u_i=\chi_1^{-1} \chi_i$ and consider the twisted character lattice $B=\langle u_1, \dots, u_{2g} \rangle$ together with its completion $B_p = \varprojlim_k B/B^{p^k}.$  Since $B$ is finitely generated and $1+p\mathcal O_L$ is
	torsion-free, $B \simeq \Z^a$ for some $a >0$ ($a \neq 0$ because only $g$ characters can coincide at most by the Hodge-Tate decomposition of $V_pJ$). Hence $B_p \simeq \Z_p^a$ is a compact abelian group in which $B$ embeds densely.
	
	As promised, we prove why twisting the character lattice is necessary, beyond the intuition that "the only meaningful way to compare multiplicities in different degrees is by normalizing all characters to have degree zero".
	\begin{remark}\label{warning}
		One might wonder why we can't take $A=\langle \chi_1, \dots, \chi_{2g} \rangle$ as our lattice instead and define analogous measures there. However, assigning $\deg(\chi_i)=1$ gives rise to a well-defined map $\deg:A\to\Z$. Indeed, assume $\theta=\prod_i \chi_i^{r_i}=\prod_i \chi_i^{s_i}.$ Let $\gamma_w \in \Gamma$ be the image of $\operatorname{Frob}_w \in G_{F_w}$ for some $w \notin S.$ Since $V_L$ is pure of weight $-1,$ $|\chi_i(\gamma_w)|_{\C}=N(w)^{-1/2}$ for all $i.$ Therefore, $\sum_i r_i = \sum_i s_i,$ and $\deg(\theta)=\sum_i r_i$ is well-defined. Define the measures $$\nu_n=\frac{1}{d_n} \sum_{\theta \in A} m_n(\theta)\delta_\theta$$ on the $p$-adic completion $A_p.$ Characters in $L_n \otimes_{\Q_p} L$ are products of $n$ degree-one characters, so $m_n(\theta) \neq 0$ implies $\deg(\theta)=n.$ Consider the residual degree map
		$$\overline{\deg}: A_p \xrightarrow{\pi_1} A/A^p \xrightarrow{\deg} \Z/p\Z.$$
		Writing $\mu_A$ for the Haar probability measure on $A_p$, for $D=\ker \left(\overline{\deg}\right)$ we have\footnote{When $p\mid n$, a residue class $rA_p^p$ can have non-zero
			mass only if $p \mid \deg r$. This makes equidistribution of multiplicities among residue classes impossible. In other words, classes
			of non-zero degree characters have zero $\nu_n$-measure for infinitely many
			$n$ even though they have positive Haar measure.}
		$$\nu_n(D) = \begin{cases}
			1, \qquad p \mid n \\
			0, \qquad p \nmid n
		\end{cases}, \qquad \mu_A(D)=\frac{1}{p}.$$
		This shows that $(\nu_n)_n$ can't weakly converge to $\mu_A.$ Now, notice that $B = \ker (\deg).$ Indeed, $B \subset \ker (\deg)$ is clear, and given $\prod_i \chi_i^{r_i}$ with $\sum_i r_i = 0,$ we can write
		$$\prod_i \chi_i^{r_i} = \chi_1^{\sum_i r_i} \prod_i \left(\chi_1^{-1}\chi_i\right)^{r_i} =\prod_i\left(\chi_1^{-1}\chi_i\right)^{r_i} \in B.$$
		This overcomes the obstruction, because $\deg(u_i)=0$ now! Hence for
		$$\overline{\deg}: B_p \rightarrow \Z/p\Z$$ one has $\ker (\overline{\deg}) = B_p$ and $\mu_n(\ker(\overline{\deg}))=1=\mu(\ker (\overline{\deg})).$ 
	\end{remark}

	As explained in the introduction, iterated brackets in our eigenbasis span $L_n\otimes_{\Q_p}L$. Their characters are
	products of the corresponding $\chi_i$, so the $G_F$-action factors
	through $\Gamma$ and every character can be written as $\chi_1^n u$ with
	$u\in B$. This allows us to write $$L_n \otimes_{\Q_p} L \simeq \bigoplus_{u \in B} \left( L_n \otimes_{\Q_p} L \right)[\chi_1^n u].$$ Thus $\sum_{u \in B} m_n(\chi_1^n u) = d_n$, and 
	$$\mu_n=\frac{1}{d_n} \sum_{u \in B} m_n(\chi_1^n u)\delta_u$$
	is a probability measure on $B_p$ for each $n$.

	We want to realize the Fourier coefficients of $\mu_n$ as normalized traces. In order to apply a symplectic character formula and control them, we extend coefficients to $\C$ and work with symplectic matrices there\footnote{It is possible to avoid working over $\C$ and consider $\mathrm{Sp}_{2g}$ over some number field instead, but the character formula we use is an identity of Laurent polynomials over $\C,$ so we believe this approach is simpler.}. Choose a (non-canonical) isomorphism $\overline{\Q}_p \simeq \C,$ inducing an embedding $\Q_p \hookrightarrow \C.$ This allows us to define $L_{n, \C}=L_n \otimes_{\Q_p} \C.$ Choose a basis of $L_n \otimes_{\Q_p} L$ consisting of iterated brackets in the fixed symplectic eigenbasis of $V_L$. Its scalar extension gives a basis of $L_{n,\C}$. Given $h \in \mathrm{Sp}_{2g}(\C),$ we denote the matrix representing the induced action on $L_{n,\C}$ by $(h \ | \ L_{n,\C}).$ Explicitly, this action is given multilinearly on the brackets, acting as $h$ on each element of $V_{\C}=L_{1,\C}.$ Notice that this action is well-defined: after writing
	$$\mathrm{gr}(U) \otimes_{\Q_p} \C \simeq \mathbb{L}(V_{\C})/(\omega), \qquad \omega=\sum_{j=1}^g[e_{x_j}, e_{y_j}],$$ we see that $$h(\omega)
	=\frac{1}{2}\sum_{i,j=1}^{2g}\left(h\Omega h^T\right)_{ij}[e_i,e_j]
	=\frac{1}{2}\sum_{i,j=1}^{2g}\Omega_{ij}[e_i,e_j]
	=\omega, \qquad \text{where} \ \  \Omega=\begin{pmatrix}
		0 & I_g \\
		-I_g & 0
	\end{pmatrix},$$ because $h$ is symplectic.
	We now prove a symplectic trace estimate. Let
	$\lambda:=g+\sqrt{g^2-1}$.
	\begin{lemma}\label{trace}
		Let $h \in \mathrm{Sp}_{2g}(\C)$ such that $h \neq \pm I$ and $h^N=I$ for some $N >0.$ Then $\Tr\left(h \ | L_{n,\C} \right) =  o(d_n).$
	\end{lemma}
	\begin{proof}
		Let $s_e=\Tr\left(h^e \ | \ V_{\C} \right).$ Since $h^N=I,$ $h$ is a diagonalizable matrix and its eigenvalues have absolute value $1.$ $h^e$ is symplectic as well, so we can organize its eigenvalues in pairs $\alpha_1, \overline{\alpha_1}, \dots, \alpha_g, \overline{\alpha_g}.$ We then have $\alpha_i + \overline{\alpha_i}=2\mathrm{Re}(\alpha_i) \in [-2,2],$ hence $s_e \in [-2g, \ 2g].$ Given $s \in \mathbb{R}$ (it will be one of the $s_e$), let $\alpha, \beta$ be the roots of $x^2-sx+1,$ and let $P_m(s)=\alpha^m+\beta^m.$ The symplectic character formula \cite[Theorem~4.3(ii)]{MSS15} in the explicit form of \cite[Proposition~4.4]{Fil13} gives
		\begin{equation}\label{inversion}
			\Tr\left(h \ | \ L_{n, \C} \right) = \frac{1}{n} \sum_{e \mid n} \mu_{\mathrm{Mob}}(e) P_{n/e}(s_e),
		\end{equation}
		where $\mu_{\mathrm{Mob}}$ is the Möbius function\footnote{We denote it by $\mu_{\mathrm{Mob}}$ to avoid confusion with the Haar measure $\mu$ on $B_p$.}. Since $|s_e| \leq 2g,$ then $|P_m(s_e)| \leq 2 \lambda^m.$ In particular, $|s_1| < 2g,$ as equality would make all its eigenvalues equal, and $h \in \lbrace I, -I \rbrace.$ Hence $\rho_h:=\max \lbrace |\alpha|, |\beta|: \alpha, \beta \ \text{roots of } x^2-s_1x+1 \rbrace < \lambda,$ and we conclude
		\begin{align*}
			\left| \Tr\left(h \ | \ L_{n, \C} \right) \right| \leq \frac{2\rho_h^n}{n}+\frac{2}{n} \sum_{\substack{e\mid n\\ e\geq 2}} \lambda^{n/e} 
			\leq \frac{2\rho_h^n}{n} + 2 \lambda^{n/2} = o(d_n).
		\end{align*}
		The last step uses $d_n \sim \frac{\lambda^n}{n},$ which is recorded in \cite[p.~123]{Kim09}. This can also be seen by choosing $h=I$ in (\ref{inversion}).
	\end{proof}
	We will use this estimate in conjunction with Weyl's criterion in the proof of Theorem \ref{conv}. The matrices we will work with are going to be diagonal in the CM eigenbasis by construction. Writing $h e_i=\lambda_i e_i,$ we see that 
	\begin{equation}\label{traceh}
		h [e_{i_1}, [e_{i_2}, \dots, e_{i_n}] \dots ]=\left(\prod_{j=1}^n \lambda_{i_j} \right) [e_{i_1}, [e_{i_2}, \dots, e_{i_n}] \dots ]
	\end{equation}
	for such matrices.
	When non-zero, the bracket above has Galois character
	$\prod_{j=1}^n\chi_{i_j}=\chi_1^n u$, with
	$u=\prod_{j=1}^n u_{i_j}$. We will exploit this in the following proof to connect traces to Fourier coefficients.
	\begin{theorem}\label{conv}
		The measures $\mu_n$ weakly converge to $\mu$.
		In particular, for every $k\geq1$ and $r\in B_p$,
		we have, for the projection
		$\pi_k:B_p\rightarrow B/B^{p^k}$,
		$$\mu_n\left(\pi_k^{-1}\left(\pi_k(r)\right)\right)
		=
		\frac{1}{d_n}
		\sum_{\substack{u\in B\\\pi_k(u)=\pi_k(r)}}
		m_n(\chi_1^n u)
		\xrightarrow{n\to\infty}p^{-ak}.$$
	\end{theorem}
	\begin{proof}
		Consider the Pontryagin dual $\widehat{B_p}=\mathrm{Hom}_{\mathrm{cts}}\left(B_p, S^1 \right).$ By Weyl's criterion for compact abelian groups \cite[Theorem B.6.3]{Kow21}\footnote{We apply it to $G=B_p$ and random variables $X_n$ satisfying $\mathbb{P}(X_n=u)=\frac{m_n(\chi_1^n u)}{d_n}.$ Hence $\mathbb{E}[\eta(X_n)]=\int_{B_p} \eta \ d\mu_n,$ and for $X$ uniformly distributed on $B_p,$ $X_n \xrightarrow{\text{law}} X$ is equivalent to $\mu_n$ weakly converging to $\mu.$}, it suffices to show, for all $1 \neq \eta \in \widehat{B_p},$ that
		$$\int_{B_p} \eta \  d\mu_n \xrightarrow{n \to \infty} 0.$$
		Fix one such $\eta.$ Since $\eta(1)=1$ and $\eta$ is continuous, there exists $k \geq 1$ such that $$\eta \left(B_p^{p^k}\right) \subset \lbrace z \in S^1: \mathrm{Re}(z)>0 \rbrace.$$ However, the only such subgroup of $S^1$ is $\lbrace 1 \rbrace,$ so $\eta(x)^{p^k}=1$ for all $x \in B_p.$ The values $\eta(u_i)$ have the same product $z^2$ on each symplectic pair, where
		$z=\sqrt{\eta\left(\chi_{\mathrm{cyc}}\chi_1^{-2}\right)}.$ We consider the matrix
		$$h_\eta=\mathrm{diag}\left(z^{-1}\eta(u_i)\right)_{i=1}^{2g}$$ acting on $V_{\C},$ written in terms of the symplectic CM eigenbasis. Clearly, $h_\eta^{2p^k}=I.$ By construction, paired entries of $h_\eta$ have product $1,$ so $h_\eta \in \mathrm{Sp}_{2g}(\C).$ If $h_\eta$ were a scalar matrix, every $\eta(u_i)$ would be equal to $\eta(u_1)=\eta(1)=1.$ This contradicts $\eta \neq 1,$ because the $u_i$ topologically generate $B_p,$ and $\eta$ is continuous. Therefore, we can apply Lemma \ref{trace} to $h_\eta$ and conclude
		\begin{equation}\label{comp}
			\int_{B_p} \eta \ d\mu_n = \frac{1}{d_n} \sum_{u \in B} m_n\left(\chi_1^n u \right) \eta(u) = \frac{z^n}{d_n} \Tr\left(h_\eta \ | \ L_{n, \C} \right) \xrightarrow{n \to \infty} 0.
		\end{equation}
		We are using (\ref{traceh}) for the second equality in (\ref{comp}): writing $\lambda_i=z^{-1}\eta(u_i),$ we see that
		$h_\eta$ acts on $L_{n,\C}[\chi_1^n u]$ via the product character
		$z^{-n}\eta(u),$ and the trace accounts for its multiplicity.

		Finally, the indicator function $\mathrm{1}_D$ of $D=\pi_k^{-1} \left( \pi_k(r)\right)=rB_p^{p^k}$ is constant on each open coset of $B_p^{p^k},$ hence continuous on $B_p$. Applying weak convergence,
		$$\mu_n(D)=\int_{B_p} \mathrm{1}_D\ d \mu_n \rightarrow \int_{B_p} \mathrm{1}_D \ d \mu = \mu(D)=p^{-ak}. $$
	\end{proof}
	
	\section{Iwasawa estimate}\label{est}
	We define the Iwasawa class module $Y_R.$ An
	annihilating relation forces every character seen by this module to satisfy a $p$-adic analytic equation, and we apply equidistribution jointly with a simple result proven in this section for analytic loci to bound the multiplicities of these exceptional characters in $L_n^\vee(1) \otimes_{\Q_p} L.$
	\subsection{Iwasawa class module}\label{classmodule}
	Fix a $\Z_p$-basis $\gamma_1, \dots, \gamma_c$ of $\Gamma.$ Define the Iwasawa algebra
	$$\Lambda:=\Z_p \llbracket \Gamma \rrbracket \simeq \Z_p \llbracket T_1, \dots, T_c \rrbracket, \qquad T_j=\gamma_j-1,$$
	and its base-changed localization
	$$R:=\mathcal{O}_L \llbracket \Gamma \rrbracket[1/p] \simeq \mathcal{O}_L \llbracket T_1, \dots, T_c \rrbracket [1/p].$$
	$R$ is a Noetherian domain. Moreover, it is an $L$-algebra because $\mathcal{O}_L[1/p]=L.$ Any continuous character $\theta: \Gamma \rightarrow 1+p \mathcal{O}_L \subset L^\times$ extends to an $L$-algebra homomorphism
	$$\theta: R \rightarrow L, \qquad T_j \mapsto \theta(\gamma_j)-1,$$ because $|\theta(\gamma_j)-1|_p < 1$ by hypothesis. Given $r \in R,$ we write $r(\theta)$ for its image under this homomorphism, to regard elements in $R$ as functions given by power series.
	
	Let $M/F_\infty$ be the maximal everywhere unramified abelian pro-$p$ extension of $F_\infty$ in which every place above $S$ splits completely, and define the class module\footnote{The possibility of taking smaller extensions than theirs is mentioned in \cite{CK10}. Ours is smaller due to extra splitting conditions. This reflects the fact that we don't need to impose splitting conditions like they do on $p$ (even though one can always find infinitely many primes satisfying them), by including an extra observation in the proof of Proposition \ref{asy}.}
	$$Y=\Gal(M/F_\infty).$$
	Since $M/F$ is Galois, there is an exact sequence 
	$$1 \rightarrow Y \rightarrow \Gal(M/F) \rightarrow \Gamma \rightarrow 1.$$
	As $Y$ is abelian, conjugation by lifts of $\gamma \in \Gamma$ to $\Gal(M/F)$ gives a continuous action of $\Gamma$ on $Y,$ making it a compact $\Lambda$-module. 
	
	By the discussion following \cite[Proposition~4.7]{Ven03}, $Y$ is a finitely generated $\Lambda$-module. Moreover, since $F(\mu_{p^\infty}) \subset F_\infty$, \cite[Corollary~4.8(i)]{Ven03} proves that $Y$ is a torsion $\Lambda$-module.

	Let $$Y_R:=Y \otimes_{\Z_p} L \simeq Y \otimes_{\Lambda} R.$$ Since $R \simeq \Lambda \otimes_{\Z_p} L$ is flat over $\Lambda,$ $Y_R$ is finitely generated and torsion over $R.$ Let $b_Y$ be the cardinality of a set of $R$-generators, and let $0 \neq \delta \in R$ be an annihilator of $Y_R,$ i.e. $$\delta Y_R =0.$$
	
	\begin{lemma}\label{iweq}
		Let $W$ be a finite-dimensional $L$-representation of the form
		$$W \simeq \bigoplus_{\theta} L(\theta)^{m(\theta)},$$
		where the $\theta: \Gamma \rightarrow 1+ p \mathcal{O}_L$ are distinct continuous characters with multiplicities $m(\theta).$ Regard $W$ as an $R$-module via the induced map $\theta: R \rightarrow L;$ $R$ acts on each summand through its character. Then
		$$\dim_L \Hom_R(Y_R,W) \leq b_Y \sum_{\delta(\theta)=0} m(\theta).$$
	\end{lemma}
	\begin{proof}
		We can write
		$$\Hom_R(Y_R,W)=\Hom_R\left( Y_R, \bigoplus_\theta L(\theta)^{m(\theta)}\right) \simeq \bigoplus_\theta \Hom_R \left(Y_R, L(\theta) \right)^{m(\theta)}.$$
		Any $\varphi \in \Hom_R(Y_R, L(\theta))$ satisfies, for all $y \in Y_R,$ $\delta(\theta)\varphi(y) = \varphi(\delta y)=\varphi(0)=0,$ meaning $\varphi = 0$ unless $\delta(\theta) = 0,$ in which case $\varphi$ is determined by its values on the $b_Y$ generators, so
		$$\dim_L \Hom_R(Y_R, L(\theta)) \leq b_Y,$$
		and the result follows from summing over the characters.
	\end{proof}
	
	\subsection{Moving analytic loci}\label{moving}
	Equidistribution
	controls the asymptotic measure of fixed residue classes, but the loci with exceptional characters vary with $n$. We deal with this inconvenience by showing a compactness result for said loci. Throughout this subsection we
	work on $\Z_p^a$ instead of $B_p$ for convenience.

	Recall that the Tate algebra $L\langle x_1, \dots, x_a, t \rangle$ consists of restricted power series, i.e. formal power series $f=\sum_{\alpha,j} c_{\alpha,j} x^\alpha t^j$ with $c_{\alpha,j} \in L$ such that $|c_{\alpha,j}| \to 0$. Any such $f$ converges uniformly on $\Z_p^a \times \Z_p,$ and defines a continuous function. The Gauss norm $\lVert f \rVert = \max |c_{\alpha,j}|$ is multiplicative, and satisfies $|f(x,t)| \leq \lVert f \rVert$ on $\Z_p^a \times \Z_p.$ Let $\mu$ be the unique Haar probability measure on $\Z_p^a.$ Then $\mu(x_0+p^k \Z_p^a)=p^{-ak}$ for all $x_0 \in \Z_p^a.$
	
	We recall the proof of Strassmann's Theorem because a similar factorization argument will be used in Lemma \ref{slice}. 
	\begin{lemma}[Strassmann]\label{strassman}
		Let $0 \neq h = \sum_{j=0}^\infty a_j t^j \in L \langle t \rangle,$ and let $N$ be the largest index with $|a_N|=\max_j |a_j|.$ Then $h$ has at most $N$ zeros in $\Z_p.$
	\end{lemma}
	\begin{proof}
		If $N=0$ then $|a_0| > \max_{j \geq 1}|a_j|,$  so $|h(t)|=|a_0|>0$ for all $t \in \Z_p.$ If $N \geq 1$ and $u \in \Z_p$ is a zero, then $h(t)= h(t)-h(u)=(t-u) \sum_{j=0}^\infty b_j t^j$ with $b_j = \sum_{k > j} a_k u^{k-j-1}.$ Hence $b_j \to 0, \ |b_j| \leq \max_{k > j} |a_k|,$ and $b_{N-1}= a_N + \sum_{k > N} a_k u^{k-N}$ satisfies $|b_{N-1}|=|a_N|,$ so $N-1$ is the largest index at which the maximum of the $|b_j|$ is attained, and the quotient has at most $N-1$ zeros by induction, so $h$ has at most $N$ zeros.
	\end{proof}
	In particular, a non-zero restricted series in one variable has finitely many zeros in $\Z_p.$ One can also prove by induction on $a$ via Fubini that a non-zero restricted power series in $a$ variables has a zero set of Haar measure zero, as recorded in the following statement.
	\begin{lemma}\label{fubini}
		Let $0 \neq h \in L\langle x_1, \dots, x_a \rangle$ and $Z(h)= \lbrace x \in \Z_p^a: h(x)=0 \rbrace.$ Then $\mu(Z(h))=0.$
	\end{lemma}
	\begin{proof}
		The proof of \cite[Lemma~4.4]{DR25} works verbatim after replacing $p=2$ by an arbitrary prime.
	\end{proof}
	We finally prove the main auxiliary lemma of this subsection, which will enable us to apply equidistribution to the exceptional character multiplicities. Its proof is rather elementary, but we couldn't find this result in the literature.
	\begin{lemma}\label{slice}
		Let $0 \neq f \in L\langle x_1, \dots, x_a,t \rangle,$ and let $(\mu_n)_n$ be a sequence of probability measures on $\Z_p^a$ that, for every fixed $k,$ $\mu_n(x_0+p^k\Z_p^a) \longrightarrow \mu(x_0+p^k\Z_p^a)$ for all $x_0 \in \Z_p^a$. Then,
		$$\mu_n \lbrace x \in \Z_p^a: f(x,n)=0 \rbrace \xrightarrow{n \to \infty} 0.$$
	\end{lemma}
	\begin{proof}
		Write $f=\sum_{\alpha} c_{\alpha}(t) x^\alpha$ with $c_\alpha \in L\langle t \rangle,$ and fix $\alpha_0$ with $c_{\alpha_0} \neq 0.$ If $f(\cdot, t_*)$ is identically zero for some $t_* \in \Z_p$, then $c_{\alpha_0}(t_*)=0,$ so by Strassmann's theorem there are only finitely many such $t_*$, denote them by $t_1, \dots, t_m.$ $c_{\alpha_0}$ vanishes at each $t_i$ with finite multiplicity at least $e_i=\min_{\alpha:\,c_\alpha\not\equiv 0}
		\operatorname{ord}_{t=t_i} c_\alpha(t).$\footnote{We only need one non-zero coefficient for each $t \in \Z_p$. Outside of $t_1, \dots, t_m, \  c_{\alpha_0}$ is non-zero. For this set, we divide by the minimum powers among all non-zero coefficients (so one of them will be non-zero at $t_i$) to ensure that the resulting function remains analytic. That's why we don't simply divide by the powers showing up in $c_{\alpha_0}.$} Then, as in the proof of Lemma \ref{strassman}, $f=(t-t_1) f_1,$ with $f_1 \in L\langle x_1, \dots, x_a, t \rangle.$ After doing this for all $i=1, \dots, m,$ we can write
		$$f=P(t)g, \ \ \text{where} \ \ P = \prod_{i=1}^m (t-t_i)^{e_i} \in L[t] \ \backslash \ \lbrace 0 \rbrace, \ \ g \in L \langle x_1, \dots, x_a, t \rangle.$$ By construction, no slice $g(\cdot,t_*)$ is identically zero. Since
		$P(n)=0$ for only finitely many integers $n$, it suffices to prove the
		result for $g$. 
		
		Assume the result is false for $g.$ In particular, there exist $\varepsilon > 0$ and infinitely many $n \in \mathbb{N}$ with $\mu_n \lbrace g(\cdot, n)=0 \rbrace \geq \varepsilon.$ Since $\Z_p$ is compact, we can find a convergent subsequence $(n_i)_i,$ and we denote its limit by $t_0 \in \Z_p.$ The zero set $Z_0(g)= \lbrace x \in \Z_p^a: g(x,t_0)=0 \rbrace$ is compact and $\mu(Z_0(g))=0$ by Lemma \ref{fubini}. Therefore, there exists an open set $Z_0(g) \subset U$ with $\mu(U) < \varepsilon/2.$ Express $U$ as a union of basis opens: $$U=\bigcup_{\alpha } x_\alpha + p^{k_\alpha} \Z_p^a.$$ Since $Z_0(g)$ is compact, we can choose a finite subcover. Moreover, two $p$-adic balls are either nested or disjoint, so we can remove the redundant ones, and get a disjoint open subcover
		$$Z_0(g) \subset \bigsqcup_{j=1}^n x_j + p^{k_j} \Z_p^a=:U_0.$$ Since $\Z_p^a \backslash U_0$ is also compact, the continuous function $|g(x, t_0)|$ attains a minimum $\eta$ there, which is positive by construction since $Z_0(g) \subset U_0.$ $g$ is continuous on the compact $\Z_p^a \times \Z_p,$ hence uniformly continuous. This means that there exists $i_0 \in \mathbb{N}$ such that $|g(x, n_i) - g(x,t_0)| < \eta$ for all $x$ and $i \geq i_0.$ If $g(x,n_i)=0$, this inequality gives $|g(x,t_0)|<\eta$, which
		is impossible outside $U_0$. Thus every zero of $g(\cdot,n_i)$ lies in
		$U_0$ for $i\geq i_0$. 
		
		We conclude $\mu_{n_i} \lbrace g(\cdot, n_i) =0 \rbrace \leq \mu_{n_i}(U_0)$ for all $i \geq i_0.$ The hypothesis about $\mu_n$ in the statement implies $\mu_{n_i}(U_0) \rightarrow \mu(U_0) < \varepsilon/2$ because $U_0$ is a disjoint union of such sets. This is a contradiction because $\mu_{n_i} \lbrace g(\cdot, n_i)=0 \rbrace \geq \varepsilon$ by assumption.
	\end{proof}
	\subsection{Main Iwasawa estimate}\label{main}
	We now write the character condition in Lemma \ref{iweq} as an analytic
	equation to which Lemma \ref{slice} applies. We define a function $f$ using $\delta$ and taking the Tate twist in the target space of the maps into account.
	
	The cyclotomic character $\chi_{\mathrm{cyc}}$ factors through $\Gamma=\Gal(F_\infty/F)$ because $\mu_{p^\infty} \subset F_{\infty},$ and it takes values in $1+p\mathcal{O}_L$ since $\mu_p \subset F.$ If $\chi_1^n u$ occurs in $L_n \otimes_{\Q_p} L,$ the corresponding character in $L_n^\vee(1) \otimes_{\Q_p} L$ is $\chi_{\mathrm{cyc}}\chi_1^{-n} u^{-1},$ and it has the same multiplicity. Fix a $\Z$-basis $m_1, \dots, m_a$ of $B.$ For the sake of clarity in the following proofs, we will use this basis to introduce coordinates on $B_p \simeq \Z_p^a.$ Notice that all the Galois characters in $L_n^\vee(1) \otimes_{\Q_p} L$ can be written as a special case of
	$$\theta_{x,t}=\chi_{\mathrm{cyc}} \chi_1^{-t} \prod_{i=1}^a m_i^{-x_i}, \qquad (x,t) \in \Z_p^a \times \Z_p.$$
	If $\psi$ is a character taking values in $1+ p \mathcal{O}_L,$ its powers are defined, for every $s \in \Z_p,$ via $\psi^s(\gamma)=\exp(s \log \psi(\gamma)).$ Hence every $\theta_{x,t}$ is continuous. We now evaluate the annihilator $\delta \in R$ of $Y_R$ at all these characters:
	$$f(x,t):=\delta(\theta_{x,t}),$$ where by evaluation we mean substituting $T_j = \theta_{x,t}(\gamma_j)-1$ in the power series $\delta.$
	We must check that $f$ is a restricted power series and that the
	substitution does not make $\delta$ vanish identically. 
	\begin{lemma}\label{rank}
		With notation as in Section \ref{class}, the matrix $\left(\log \chi_i(\gamma_j)\right)_{j,i} \in M_{c,2g}(L)$ has rank $c.$
	\end{lemma}
	\begin{proof}
		$\lbrace \log \gamma_1, \dots, \log \gamma_{c} \rbrace$ is a $\Z_p$-basis of $\log \Gamma,$ so it is linearly independent over $\Q_p$ in $K_p$, hence over $L$ after the base change $K_p \otimes_{\Q_p} L \simeq L^{2g},$ where the image of $\log \gamma_j$ is 
		$$(\iota_i(\log \gamma_j))_i=(\log \chi_i(\gamma_j))_i.$$
		This equality follows from the fact that the $p$-adic logarithm commutes with continuous ring homomorphisms by \cite[Theorem~45.13]{Sch84}\footnote{Recall that in Section \ref{class}, $\iota_i: K_p \rightarrow L$ factors through one of the field components of $K_p.$ We can extend $\iota_i$ to a $\Q_p$-automorphism of $\overline{\Q}_p,$ which gives rise to a continuous $\Q_p$-automorphism of $\C_p$, and the cited result applies.}. These are the rows of the matrix, so  $\mathrm{rank}_L\left(\log \chi_i(\gamma_j)\right)_{j,i}=c.$
	\end{proof}
	We now prove the intuitive result that, since $\delta \neq 0,$ the function $f$ is a non-zero restricted series.  
	\begin{lemma}\label{nonzero}
		$f \in L \langle x_1, \dots, x_a, t \rangle$ and it is non-zero.
	\end{lemma}
	\begin{proof}
		We first prove it is restricted.  For $1 \leq j \leq c,$ let
		$$\ell_j(x,t):=-t\log \chi_1(\gamma_j)-\sum_{i=1}^a x_i \log m_i(\gamma_j), \qquad c_j:= \chi_{\mathrm{cyc}}(\gamma_j).$$ The coefficients of $\ell_j$ lie in $p \mathcal{O}_L$ and $c_j \in 1+ p \mathcal{O}_L,$ so 
		$$s_j(x,t):=c_j \exp(\ell_j(x,t))-1$$ is a restricted power series with $\lVert s_j \rVert \leq |p| < 1.$
		Write $\delta \in R$ as $\delta=\sum_\beta d_\beta T^\beta,$ with bounded coefficients. Then
		$$\lVert d_\beta s^\beta \rVert \leq \lVert \delta \rVert |p|^{|\beta|} \to 0,$$ and $f=\sum_\beta  d_\beta \prod_{j=1}^c s_j^{\beta_j}= \sum_\beta d_\beta s^\beta \in L\langle x_1, \dots, x_a,t \rangle.$ 
		
		Let us now prove that $f \neq 0.$ Consider the automorphism 
		$$\tau: R \rightarrow R, \qquad T_j \mapsto c_j(1+T_j)-1,$$ whose inverse replaces $c_j$ by $c_j^{-1}.$ A direct substitution shows $$f(x,t)=\delta(\theta_{x,t})=\tau(\delta)(\exp(\ell_1(x,t))-1, \dots, \exp(\ell_c(x,t))-1).$$
		By Lemma \ref{rank}, $(\log \chi_i(\gamma_j))_{j,i}$ has rank $c.$ Since $\langle \chi_1, m_1, \dots, m_a \rangle_{\Z} = \langle \chi_1, \dots, \chi_{2g} \rangle_{\Z},$ their logarithms span the same $L$-vector space, and the linear map $L^{a+1} \rightarrow L^c$ given by
		$$(x,t) \mapsto (\ell_j(x,t))_j$$ is surjective. Choose a section $L^c \rightarrow L^{a+1}$ given by $z \mapsto (x(z), t(z)).$ Writing everything in terms of power series, we get
		$$f(x(z), t(z))=\tau(\delta)(\exp(z_1)-1, \dots, \exp(z_c)-1).$$ Since $$L\llbracket T_1, \dots, T_c \rrbracket \rightarrow L \llbracket z_1, \dots, z_c \rrbracket, \qquad T_j \mapsto \exp(z_j)-1$$ is an isomorphism (with inverse $z_j \mapsto \log(1+T_j)),$ it sends the non-zero series $\tau(\delta)$ to a non-zero series. Hence $f(x,t) \neq 0.$ 
	\end{proof}
	Finally, we show the main Iwasawa estimate. The role of equidistribution becomes clear with this application; it enables us to work with the full unipotent quotient $U_n.$
	\begin{theorem}\label{estimate}
		We have
		$$\frac{\dim_L \Hom_R\left(Y_R, L_n^\vee(1) \otimes_{\Q_p} L\right)}{d_n} \xrightarrow{n \to \infty} 0.$$
	\end{theorem}
	\begin{proof}
		By Lemma \ref{iweq} applied to the character decomposition of $L_n^\vee(1) \otimes_{\Q_p} L$,
		\begin{align*}
			\dim_L \Hom_R(Y_R, L_n^\vee(1) \otimes_{\Q_p} L) &\leq b_Y \sum_{u: f(u, n)=0} m_n(\chi_1^n u) \\
			&= b_Y d_n \cdot \mu_n \lbrace u \in B_p: f(u,n)=0 \rbrace.
		\end{align*}
		By Lemma \ref{nonzero}, $ 0 \neq f \in L\langle x_1, \dots, x_a, t \rangle$, so Lemma \ref{slice} gives $\mu_n \lbrace f(\cdot, n)=0 \rbrace \xrightarrow{n \to \infty} 0,$ and the result follows after dividing both sides by $d_n.$
	\end{proof}
	The connection with cohomology is similar to the one in \cite{CK10} now. The main difference is the way we bound the representations involved (because the dimensions of $L_n$ grow exponentially), and the fact that we work over number fields.
	\section{Cohomology estimates and dimension comparison}\label{cohom}
	We first bound a "Tate-Shafarevich" group using the Iwasawa class module, then control $H^2$ via 
	purity combined with Poitou-Tate duality, and finally apply the Euler characteristic formula together 
	with a trace estimate for complex conjugation to bound $H^1.$ This gives an estimate for 
	$H^1_f(G_{k,T}, U_n).$ For the local target $\prod_{v\mid p}H^1_f(G_v,U_n)$, we compare (Lie 
	algebras of) Hodge filtrations on (Lie algebras of) unipotent groups at each $v\mid p$.
	\subsection{From Iwasawa to global Galois cohomology}
	\begin{lemma}\label{vanish}
		For every finite place $v$ of $k,$
		$$H^0(G_v, L_n)=0 \ \ \ (n \geq 1), \qquad H^0(G_v, L_n^\vee(1))=0 \ \ \ (n \geq 3).$$
		In particular, $H^0(G_{k,T}, L_n)=0$ for all $n \geq 1.$
	\end{lemma}
	\begin{proof}
		If $v \nmid p,$ potential good reduction of CM abelian varieties \cite[Theorem~6(a)]{ST68} gives a finite extension $E/k_v$ over which $J$ has good reduction, making $V_pJ|_{G_E}$ unramified over $E.$ The geometric Frobenius $\operatorname{Frob}_E$ has eigenvalues of complex absolute value $q_E^{-1/2},$ so the quotient $\left(V_pJ|_{G_E}\right)^{\otimes n} \twoheadrightarrow \left(L_{n}\right)|_{G_E}$ is pure of weight $-n.$ Hence $L_{n}^\vee(1)|_{G_E}$ is pure of weight $n-2.$ If we denote the cardinality of the residue field of $E$ by $q_E,$ it is clear that $q_E^{-n/2} < 1$ for all $n \geq 1$ and $q_E^{(n-2)/2} > 1$ for $n \geq 3.$ Then Frobenius doesn't fix any non-zero vector, and both local vanishing results follow. 
		
		At $v \mid p,$ $J$ has good reduction. Let $q_v=p^{f_v}$ be the cardinality of the residue field of $k_v.$ Crystalline Frobenius $\varphi^{f_v}$ has eigenvalues of complex absolute value $q_v^{(n-2)/2}$ and $q_v^{-n/2}$ on $D_{\cris,v}(L_n^\vee(1))$ and $D_{\cris,v}(L_n)$ respectively. Both local vanishings follow from $W^{G_v} \hookrightarrow D_{\cris,v}(W)^{\varphi=1}$. Finally, for any $v \mid p,$ 
		$$H^0(G_{k,T}, L_n)=L_n^{G_{k,T}} \subset L_n^{G_v} = H^0(G_v, L_n)=0.$$
	\end{proof}
	Given a finite-dimensional $\Q_p$-representation $W$ of $G_{k,T},$ let
	$$\Sh_T^1(W)=\Ker \left(H^1(G_{k,T}, W) \rightarrow \bigoplus_{v \in T} H^1(k_v, W) \right),$$ and define $\Sh_{F,S}^1(W)$ in a similar fashion. At archimedean places, positive degree cohomology vanishes because the local Galois group is finite and its order is invertible in $\Q_p$.
	\begin{prop}\label{asy}
		For $n \geq 3,$ restriction induces an injection
		\begin{equation*}
			\Sh_{F,S}^1(L_n^\vee(1) \otimes_{\Q_p} L) \hookrightarrow \Hom_R(Y_R, L_n^\vee(1) \otimes_{\Q_p} L).
		\end{equation*}
		Therefore,
		$$\frac{\dim_{\Q_p} \Sh_T^1(L_n^\vee(1))}{d_n} \xrightarrow{n \to \infty} 0.$$
	\end{prop}
	\begin{proof}
		Fix $n \geq 3,$ and let $H=\Ker(G_{F,S} \rightarrow \Gamma)=\Gal(F_S/F_\infty).$ It acts trivially on $L_n^\vee(1)$ because it fixes both $L_n$ and $\Q_p(1).$ We first claim that $H^1(\Gamma, L_n^\vee(1) \otimes_{\Q_p} L)=0.$ To see this, choose a place $w \notin S$ and let $\gamma_0 \in \Gamma$ be the image of geometric Frobenius via $G_{F_w} \rightarrow G_F \rightarrow G_{F,S} \rightarrow \Gamma.$ Its eigenvalues on $L_n^\vee(1)$ have absolute value $N(w)^{(n-2)/2} > 1,$ so $\gamma_0-1$ is invertible as an endomorphism of $L_n^\vee(1) \otimes_{\Q_p} L.$ Given $[c] \in H^1(\Gamma, L_n^\vee(1) \otimes_{\Q_p} L),$ applying the cocycle condition to $c(\gamma \gamma_0)=c(\gamma_0 \gamma)$ (since $\Gamma$ is abelian) and comparing them gives
		$c(\gamma)=(\gamma-1) \cdot \left((\gamma_0-1)^{-1} c(\gamma_0)\right),$ so $c$ is a coboundary and $H^1(\Gamma, L_n^\vee(1) \otimes_{\Q_p} L)=0.$
		
		Inflation-restriction applied to $1 \rightarrow H \rightarrow G_{F,S} \rightarrow \Gamma \rightarrow 1$ then gives an injection
		\begin{equation}\label{inj}
			H^1(G_{F,S}, L_n^\vee(1) \otimes_{\Q_p} L) \hookrightarrow H^1(H, L_n^\vee(1) \otimes_{\Q_p} L)^\Gamma = \Hom_{\mathrm{cts}}(H, L_n^\vee(1) \otimes_{\Q_p} L)^\Gamma.
		\end{equation}
		Here the $\Gamma$ superscript means $\Gamma$-equivariant. For the equality, we are using that $H$ acts trivially on $L_n^\vee(1) \otimes_{\Q_p} L,$ so the cocycle condition is equivalent to the maps being group homomorphisms. Given $[c] \in \Sh_{F,S}^1(L_n^\vee(1) \otimes_{\Q_p} L),$ consider its restriction $\phi=c|_H: H \rightarrow L_n^\vee(1) \otimes_{\Q_p} L.$ By construction, $\loc_v([c])=0$ for all $v \in S,$ so there exists $a_v \in L_n^\vee(1) \otimes_{\Q_p} L$ such that $c(\sigma)=(\sigma-1)a_v$ for all $\sigma \in D_v=\mathrm{Im}(G_{F_v} \rightarrow G_{F,S}).$ Therefore, $\phi$ vanishes on $D_v \cap H$ by trivial action of $H.$ By $\Gamma$-equivariance, $\phi$ also vanishes in all of its conjugate subgroups. Therefore, we claim that the extension $F_S^{\Ker \phi}/F_\infty,$ which is unramified outside of $S$ (because $F_S^{\Ker \phi} \subset F_S$), is completely split above $S.$ Indeed, if $w$ is a place of $F_\infty$ above $ v \in S,$ let $\widetilde{w}$ be a place on $F_S$ above it. Notice that, up to conjugacy, $D_{\widetilde{w}}(F_S/F_\infty) = D_{v} \cap H \subset \Ker \ \phi.$ Write $w^\circ=\widetilde{w}|_{F_S^{\Ker \phi}}.$ Under the projection $H \rightarrow H/{\Ker \  \phi},$ $$D_{w^\circ}\left(F_S^{\Ker \ \phi}/F_\infty \right)=\mathrm{Im}(D_{\widetilde{w}} \rightarrow H/\Ker \ \phi)=1,$$ as we wanted to show. Moreover, $\phi(H)$ is compact and lies in a $\Z_p$-lattice, so it is an abelian pro-$p$ group. Hence $F_S^{\Ker  \phi} \subset M.$ Recall that $M/F_\infty$ is the maximal abelian pro-$p$ extension everywhere unramified and totally split above $S;$ recall also our notation $Y=\Gal(M/F_\infty).$ The restriction map
		$$r: H=\Gal(F_S/F_\infty) \rightarrow Y = \Gal(M/F_\infty)$$ satisfies $$\Ker \  r = \Gal(F_S/M) \subset \Ker \ \phi = \Gal(F_S/F_S^{\ker \phi}).$$  Consequently, $\phi$ factors through $Y,$ and we write $\phi = \bar{\phi} \circ r,$ with $\bar{\phi}: Y \rightarrow L_n^\vee(1) \otimes_{\Q_p} L.$ Notice that $\bar{\phi}$ is continuous and $\Gamma$-equivariant by (\ref{inj}), hence $\Lambda=\Z_p\llbracket \Gamma \rrbracket$-linear (because it is a group homomorphism as well). The target is an $R$-module because $R \simeq \Lambda \otimes_{\Z_p} L$, so $\bar{\phi}$ extends uniquely to an $R$-linear map $\bar{\phi}_R: Y_R \rightarrow L_n^\vee(1) \otimes_{\Q_p} L.$ If $\bar{\phi}_R=0,$ $\bar{\phi}=0,$ so $c|_H=\bar{\phi} \circ r = 0.$ Since (\ref{inj}) is injective, this implies $[c]=0.$ Hence $[c] \mapsto \bar{\phi}_R$ defines an $L$-linear injection
		\begin{equation*}
			\Sh_{F,S}^1(L_n^\vee(1) \otimes_{\Q_p} L) \hookrightarrow \Hom_R(Y_R, L_n^\vee(1) \otimes_{\Q_p} L).
		\end{equation*}
		Now Theorem \ref{estimate} gives
		$$\dim_L \Sh_{F,S}^1(L_n^\vee(1) \otimes_{\Q_p}L)=o(d_n).$$
		Finally, we descend this estimate to $k.$ 
		
		Fix $n \geq 3$. Recall that $T'$ is the union of $T$ and the places ramified in $F/k$. At $v \in T'\backslash T$ the curve has good reduction and a class in $H^1(G_{k,T},L_n^\vee(1))$ is unramified at $v,$ so its localization lies in $H^1_{\mathrm{ur}}(k_v, L_n^\vee(1)) = L_n^\vee(1)/(\mathrm{Frob}_v -1)L_n^\vee(1),$ which is zero because $\mathrm{Frob}_v$ has no eigenvalue $1$ on $L_n^\vee(1)$ for $n \geq 3.$ Hence $\Sh_T^1(L_n^\vee(1)) \hookrightarrow \Sh_{T'}^1(L_n^\vee(1)).$ Moreover, the restriction map $\mathrm{Res}: H^1(G_{k,T'}, L_n^\vee(1)) \rightarrow H^1(G_{F,S}, L_n^\vee(1))$ is injective, because $F_S=k_{T'}$ gives $G_{F,S}=\Gal(k_{T'}/F) \subset G_{k,T'},$ hence $\mathrm{Cor} \circ \mathrm{Res} = [F:k] \cdot \mathrm{Id}$ (see e.g. \cite[VIII, \S~2.]{Ser79}), which is invertible on a $\Q_p$-vector space. So we have injections
		$$\Sh_T^1(L_n^\vee(1)) \hookrightarrow \Sh_{T'}^1(L_n^\vee(1)) \hookrightarrow \Sh_{F,S}^1(L_n^\vee(1)).$$ Since $L/\Q_p$ is finite, we can view cocycles in coordinates, and it follows that $- \otimes_{\Q_p} L$ commutes with local and global $H^1.$ It also commutes with kernels by exactness, giving
		$$\Sh_T^1(L_n^\vee(1)) \otimes_{\Q_p} L \simeq \Sh_T^1(L_n^\vee(1) \otimes_{\Q_p} L) \hookrightarrow \Sh_{F,S}^1(L_n^\vee(1) \otimes_{\Q_p} L).$$ We conclude that
		$$\dim_{\Q_p} \Sh_T^1(L_n^\vee(1)) \leq \dim_L \Sh_{F,S}^1(L_n^\vee(1) \otimes_{\Q_p} L) = o(d_n).$$
	\end{proof}
	\subsection{Dimension comparison}
	Let $h^i(G_{k,T}, L_n):= \dim_{\Q_p} H^i(G_{k,T}, L_n).$
	\begin{prop}\label{odn}
		As $n \to \infty$,
		$$h^2(G_{k,T}, L_n)=o(d_n), \qquad h^1(G_{k,T}, L_n) = \frac{[k:\Q]}{2} d_n + o(d_n).$$
	\end{prop}
	\begin{proof}
		For $n \geq 3$ and finite $v \in T,$ local Tate duality and Lemma \ref{vanish} give
		$$H^2(G_v, L_n) \simeq H^0(G_v, L_n^\vee(1))^\vee=0.$$
		At archimedean places, it is also zero. It follows that $\Ker \left( H^2(G_{k,T}, L_n) \rightarrow \bigoplus_{v \in T} H^2(G_v, L_n) \right)=H^2(G_{k,T}, L_n).$ As in the proof of \cite[Theorem~0.1]{CK10}, Poitou-Tate duality \cite[Theorem 4.10]{Mil86} consequently gives
		$$H^2(G_{k,T}, L_n) \simeq \left(\Sh_T^1(L_n^\vee(1))\right)^\vee.$$ By Proposition \ref{asy}, it follows that $h^2(G_{k,T}, L_n)=o(d_n).$

		Let $r_1$ and $r_2$ be the numbers of real and complex places of $k,$ so that $r_1+2r_2=[k:\Q].$ For each real place $v,$ let $c_v$ be its complex conjugation. Choose a basis of $V_pJ$ where the Weil pairing's matrix form is $$\Omega=\begin{pmatrix}
			0 & I_{g} \\
			-I_{g} & 0
		\end{pmatrix},$$ and let $C_v$ be the matrix of $c_v$ in that basis. As before, we write $(C_v \ | \ L_n)$ for the induced matrix. Lemma \ref{vanish} combined with the global Euler characteristic formula \cite[Lemma~2]{Jan89} reads
		\begin{align*}
			h^1(G_{k,T}, L_n) &= h^2(G_{k,T}, L_n) + r_2 d_n+ \sum_{v \ \text{real}} \dim L_n^{c_v=-1} \\
			&=h^2(G_{k,T}, L_n)+ \frac{[k:\Q]}{2}d_n - \frac{1}{2} \sum_{v \ \text{real}} \Tr\left(C_v \ | \ L_n \right).
		\end{align*}
		We have $C_v^T \Omega C_v = -\Omega$ and $C_v^2=I,$ so after extending scalars to $\C,$ $h_{c_v}=-iC_v \in \mathrm{Sp}_{2g}(\C)$ with $h_{c_v} \neq \pm I, \ h_{c_v}^4=I,$ and Lemma \ref{trace} gives
		$$\Tr(C_v \ | \ L_n)=i^n\Tr(h_{c_v} \ | \ L_{n,\C})=o(d_n).$$ We conclude $$h^1(G_{k,T}, L_n)=\frac{[k:\Q]}{2}d_n + o(d_n).$$
	\end{proof}
	\begin{theorem}\label{thm}
		As $n\to\infty$,
		\begin{align*}
			\dim_{\Q_p}H^1_f(G_{k,T},U_n)
			&\leq\frac{[k:\Q]}{2}D_n+o(D_n),\\
			\dim_{\Q_p}\prod_{v\mid p}
			\Res_{k_v/\Q_p} U_{n,v}^{\dR}/F^0
			&=[k:\Q]D_n-o(D_n).
		\end{align*}
		In particular, for all $n\gg0$,
		$$\dim_{\Q_p}H^1_f(G_{k,T},U_n)
		<
		\dim_{\Q_p}\prod_{v\mid p}
		\Res_{k_v/\Q_p} U_{n,v}^{\dR}/F^0,
		\qquad
		\limsup_{n\to\infty}
		\frac{\dim_{\Q_p}H^1_f(G_{k,T},U_n)}
		{\dim_{\Q_p}\prod_{v\mid p}
			\Res_{k_v/\Q_p} U_{n,v}^{\dR}/F^0}
		\leq\frac12.$$
	\end{theorem}
	\begin{proof}
		By \cite[Corollary~2.2.7]{Bet23} and the proof of \cite[Proposition~3.2.4]{Bet23} respectively, there are closed immersions of schemes
		$$H^1(G_{k,T}, U_n) \hookrightarrow \prod_{j=1}^n H^1(G_{k,T}, L_j), \qquad H^1_f(G_{k,T}, U_n) \hookrightarrow H^1(G_{k,T}, U_n).$$ Proposition \ref{odn} hence gives the global estimate
		$$\dim_{\Q_p} H^1_f(G_{k,T}, U_n) \leq \sum_{j=1}^n h^1(G_{k,T}, L_j) = \frac{[k:\Q]}{2} D_n + o(D_n). $$
		On the local side, fix a place $v\mid p$. In what follows, we work with de Rham dimensions over $k_v$, and $D_{\dR,v}$ denotes the de Rham functor for $G_v$. The group schemes $U_{n,v}^{\dR}$ and $F^0U_{n,v}^{\dR}$ are smooth algebraic groups in characteristic zero, hence have the same dimensions as their Lie algebras. Olsson's comparison theorem \cite[Theorem~1.8]{Ols11} in the form of \cite[Sections~1.1--1.2]{Ols16} gives $D_{\dR,v}(U_n)\simeq U_{n,v}^{\dR}$ and shows that $U_n$ is crystalline at $v$. By \cite[Definition-Lemma~4.14]{Bet24} applied to $B=B_{\cris}$, $\mathrm{Lie}\ U_n$ is also crystalline, hence de Rham. The Hodge filtration on
		$$D_{\dR,v}(V_pJ)\simeq H^1_{\dR}(X_{k_v})^\vee$$
		satisfies $\dim_{k_v}F^0D_{\dR,v}(V_pJ)=g$ and $F^1D_{\dR,v}(V_pJ)=0$, so \cite[Proposition~6.3.3]{BC09} gives a quotient map
		$$\left(F^0D_{\dR,v}(V_pJ)\right)^{\otimes j}
		=F^0D_{\dR,v}(V_pJ^{\otimes j})\twoheadrightarrow F^0D_{\dR,v}(L_j).$$
		Hence $\dim_{k_v}F^0D_{\dR,v}(L_j)\leq g^j$. Since $d_n\sim\lambda^n/n$ \cite[p.~123]{Kim09} and $\lambda>g$,
		$$\sum_{j=1}^n\dim_{k_v}F^0D_{\dR,v}(L_j)\leq\sum_{j=1}^ng^j=O(g^n)=o(D_n).$$
		\cite[Corollary~8.13]{Ols16} identifies $F^0D_{\dR,v}(U_n)\simeq F^0U_{n,v}^{\dR}$, and differentiation at the identity gives
		$$\mathrm{Lie}\left(F^0D_{\dR,v}(U_n)\right)\simeq\mathrm{Lie}\left(F^0U_{n,v}^{\dR}\right).$$
		On the other hand, \cite[Lemma~4.13]{Bet24}, for $B=B_{\dR}^+$ identifies the left side with $F^0D_{\dR,v}(\mathrm{Lie}\ U_n)$. Combining both isomorphisms, we conclude
		$$\mathrm{Lie}\left(F^0U_{n,v}^{\dR}\right)\simeq F^0D_{\dR,v}(\mathrm{Lie}\ U_n).$$
		Using \cite[Proposition~6.3.3]{BC09} inductively on $0\longrightarrow L_j\longrightarrow\mathrm{Lie}\ U_j\longrightarrow\mathrm{Lie}\ U_{j-1}\longrightarrow0$ and applying $F^0$, we obtain
		\begin{align*}
			\dim_{k_v}F^0U_{n,v}^{\dR}
			&=\dim_{k_v}F^0D_{\dR,v}(\mathrm{Lie}\ U_n)\\
			&=\sum_{j=1}^n\dim_{k_v}F^0D_{\dR,v}(L_j)=o(D_n).
		\end{align*}
		Since $\dim_{k_v}U_{n,v}^{\dR}=D_n$, \cite[Lemma~2.1.11]{Bet23} gives
		$$\dim_{k_v}(U_{n,v}^{\dR}/F^0)=D_n-o(D_n).$$
		Hence the Weil restriction satisfies
		\begin{align*}
			\dim_{\Q_p}\prod_{v\mid p}\Res_{k_v/\Q_p}(U_{n,v}^{\dR}/F^0)
			&=\sum_{v\mid p}[k_v:\Q_p]\dim_{k_v}(U_{n,v}^{\dR}/F^0)\\
			&=[k:\Q]D_n-o(D_n).
		\end{align*}
		The remaining claims readily follow.
	\end{proof}

\end{document}